\documentclass[12pt]{article}

\usepackage{amsmath,amssymb,amsbsy,amsthm,latexsym, amsopn, amstext, euscript, amscd, epsf, graphicx}
\usepackage{hyperref}
\usepackage{geometry}

\usepackage[OT2,T1]{fontenc}
\hypersetup{colorlinks=true,citecolor=blue,urlcolor=blue,pdfnewwindow=true}

\begin{document}

\newtheorem{theorem}{Theorem}
\newtheorem{lemma}{Lemma}
\newtheorem{corollary}{Corollary}
\newtheorem*{conjecture}{Conjecture}

\def\cA{\mathcal{A}}
\def\cJ{\mathcal{J}}
\def\cO{\mathcal{O}}
\def\cS{\mathcal{S}}
\def\cT{\mathcal{T}}
\def\C{\mathbb{C}}
\def\R{\mathbb{R}}
\def\Z{\mathbb{Z}}
\def\le{\leqslant}
\def\ge{\geqslant}
\def\({\left(}
\def\){\right)}

\newcommand{\comm}[1]{\marginpar{
\vskip-0.7\baselineskip
\raggedright\footnotesize
\itshape\hrule\smallskip#1\par\smallskip\hrule}}

\def\xxx{\vskip5pt\hrule\vskip5pt}
\def\imhere{\xxx\centerline{\sc I'm here!}\xxx}
\def\finishline{\xxx\centerline{\sc $\cdots\cdots$ finish line $\cdots\cdots$}\xxx}

\title{\sc Sums and products
\break with smooth numbers}

\author{
{\sc William D.~Banks } \\
{Department of Mathematics}  \\
{University of Missouri} \\
{Columbia, MO 65211 USA} \\
{\tt BanksWD@missouri.edu }  \\
\\
{\sc  David J.~Covert} \\
{Department of Mathematics} \\
{University of Missouri} \\
{Columbia, MO 65211 USA} \\
 {\tt CovertDJ@missouri.edu} \\}

\date{\today}

\maketitle

\begin{abstract}
We estimate the sizes of the sumset $\cA+\cA$ and the productset
$\cA\cdot\cA$ in the special case that $\cA=\cS(x,y)$, the set of
positive integers $n\le x$ free of prime factors exceeding $y$.
\end{abstract}

\pagebreak

\section{Background}

For any nonempty subset $\cA$ of a ring, the {\it sumset} and {\it productset} of $\cA$ are defined as 
\[
\cA + \cA = \{a + a' : a, a' \in \cA\}
\qquad\text{and}\qquad
\cA \cdot \cA = \{a \cdot a' : a, a' \in \cA\},
\]
respectively.  A famous problem of Erd\H os and Szemer\'edi \cite{ES}
asks one to show that the sumset and productset of a finite set
of integers cannot both be small.

\begin{conjecture}
\text{\rm (Erd\H os--Szemer\'edi)}
For any fixed $\delta>0$ the lower bound
\[
\max\{|\cA+\cA|, |\cA \cdot \cA| \}\,\mathop{\gg}\limits_{\delta}\,|\cA|^{2 - \delta}
\]
holds for all finite sets $\cA\subset\Z$.
\end{conjecture}

Erd\H os and Szemer\'edi \cite{ES} took the first step towards this conjecture
by showing that for some $\epsilon>0$, one has a lower bound of the form
\begin{equation}
\label{eq:SPineq}
\max\{ |\cA + \cA|, |\cA \cdot \cA| \} \ge c(\epsilon)\,|\cA|^{1 + \epsilon}
\end{equation}
for all finite sets $\cA\subset\Z$.
Nathanson \cite{Nath} gave the first explicit bound by showing that one can take
$\epsilon=\frac{1}{31}$ and $c(\epsilon) = 0.00028\cdots$ in this inequality, and later,
Ford~\cite{Ford} showed that $\epsilon=\frac{1}{15}$ is acceptable.
Establishing an important connection between the sum-product problem and geometric incidence theory, Elekes~\cite{Elekes} showed that one can take $\epsilon=\frac{1}{4}$
via a clever application of the the Szemer\'edi--Trotter incidence theorem
(which counts incidences between points and lines in the plane); moreover,
his argument readily extends to finite sets of real numbers.
Further improvements, including the best known bound to date, have been given by
Solymosi~\cite{Soly05, Soly09}; he has shown that~\eqref{eq:SPineq}
holds with any $\epsilon<\frac{1}{3}$ for all finite sets $\cA\subset\R$.

Although the Erd\H os--Szemer\'edi conjecture remains open, it is known
that the productset must be large whenever the sumset is sufficiently small.
In fact, Nathanson and Tenenbaum~\cite{NT} have shown that 
\begin{equation} \label{NaTen}
|\cA \cdot \cA| \ge \frac{c\,|\cA|^2}{\log |\cA|}\qquad \text{if}
\quad |\cA + \cA| \le 3 |\cA| - 4.
\end{equation}
The aforementioned best known bound to date,
given by Solymosi \cite{Soly09}, follows from his more general inequality
\begin{equation} \label{energy}
|\cA + \cA|^2 |\cA \cdot \cA| \ge \frac{|\cA|^4}{4 \lceil \log |\cA| \rceil}\,.
\end{equation}
Note that \eqref{energy} provides a
quantitive generalization of the Nathanson--Tenenbaum result~\eqref{NaTen}
(see also the results in~\cite{Elekes,ER,Soly05}); it implies that
$|\cA \cdot \cA| \ge |\cA|^{2 - \delta_{\epsilon}}$
whenever $|\cA+\cA| < |\cA| ^{1 + \epsilon}$, where $\delta_{\epsilon} \to 0$ as $\epsilon \to 0$.

In the opposite direction, Chang~\cite{Chang03} has shown that the sumset must be large
whenever the productset is sufficiently small.  More precisely, she has shown that
\begin{equation} \label{chang}
|\cA + \cA| > 36^{-\alpha} |\cA|^2 \qquad \text{if}
\qquad|\cA \cdot \cA| < \alpha|\cA|\quad\text{~for some constant~}\alpha.
\end{equation}


A great deal of attention has also been given to the sum-product problem in other rings,
including (but not limited to) finite fields, polynomial rings, and matrix rings.
For a thorough account of the subject, we refer the reader to~\cite{Tao} and the references contained therein.

\section{Statement of results}
Let $\Omega$ be any infinite collection of finite sets within a given ring.
We shall say that $\Omega$ has the {\it Erd\H os--Szemer\'edi property} if
\[
\max\{|\cA + \cA| , |\cA \cdot \cA|\} = |\cA|^{2 + o(1)} \qquad \text{as}\quad
 |\cA| \to \infty \text{ with } \cA \in \Omega.
\]
Then, the Erd\H os--Szemer\'edi conjecture is the assertion that the collection
consisting of \emph{all} finite sets of integers has the Erd\H os--Szemer\'edi property.

In this paper, we study the Erd\H os--Szemer\'edi property with collections of sets of
{\it smooth numbers}, i.e., sets of the form
\[
\cS(x,y) = \{n \le x : P^+(n) \le y\} \qquad (x \ge y \ge 2),
\]
where $P^+(n)$ denotes the largest prime factor of an integer $n \ge 2$, and $P^+(1) = 1$.
These sets are well known in analytic number theory; for a background on integers free of large
prime factors, we refer the reader to \cite[Chapter III.5]{Ten} (see also the
survey~\cite{Gran}).

\begin{theorem} 
There is an absolute constant $c > 0$ for which the collection
\[
\Omega = \left\{ \cS(x,y) : 2 \le y \le c \log x \right\}
\]
has the Erd\H os--Szemer\'edi property.
\end{theorem}

\noindent {\sc Remark.} For smaller values of $y$ of size $o(\log x)$ we show that the productset of
$\cA = S(x,y)$ has size $|\cA|^{1 + o(1)}$ (see Theorem~\ref{prodsmall}), and thus
only the sumset is large in this region.

\begin{theorem}
Let $f$ be an arbitrary real-valued function such that $f(x) \to \infty$ as $x \to \infty$.  Then, the collection
\[
\Omega = \left\{ \cS(x,y) : f(x) \log x \le y \le x \right \}
\]
has the Erd\H os--Szemer\'edi property.
\end{theorem}

\noindent{\sc Remark.} For slightly larger values of $y$ exceeding $(\log x)^{f(x)}$ we show that the sumset of $\cA = \cS(x,y)$ has size $|\cA|^{1 + o(1)}$ (see Theorem~\ref{largesum}), and
hence only the productset is large in this region.

\bigskip

Since each set $\cS(x,y)$ is multiplicatively defined, it is quite difficult to estimate
the size of the sumset $\cS(x,y) + \cS(x,y)$ for values of $y$ close to $\log x$.
It is reasonable to expect that for every fixed $\kappa > 0$ one has
\[
|\cS(x,y) + \cS(x,y)| = |\cS(x,y)|^{2 + o(1)} \qquad (x \to \infty,~y = \kappa \log x).
\]
In view of \eqref{kprod}, the Erd\H os--Szemer\'edi conjecture implies that this
is true.  A partial result in this direction is provided by \eqref{ksum}.
We also expect that for any fixed $A > 1$ one has
\[
|\cS(x,y) + \cS(x,y)| = |\cS(x,y)|^{\beta_A + o(1)} \qquad \(x \to \infty,~y= (\log x)^A\)
\]
for some constant $\beta_A$ in the open interval $(1,2)$.  For
$A > 2$, a partial result in this direction is provided by Theorem~\ref{A}.

\

\noindent {\bf Acknowledgements.} The authors would like to thank Derrick Hart, Alex Iosevich, and Igor Shparlinski for helpful conversations.

\section{Preliminaries}

As before, we write
\[
\cS(x,y) = \left \{n \le x : P^+(n) \le y \right\} \qquad (x\ge y\ge 2),
\]
and we now set
\[
\Psi(x,y) = |\cS(x,y)| \qquad (x\ge y\ge 2).
\]

We also put
\[
G(t) = \log(1 + t) + t \log(1 + t^{-1}) \qquad (t > 0).
\]
From this definition we immediately derive the crude estimates
\begin{equation} \label{tlarge}
G(t) = \log t \left\{ 1 + O\( \frac{1}{\log t} \) \right\} \qquad ( t \ge 2)
\end{equation}
and
\begin{equation} \label{tsmall}
G(t) = t \log t^{-1} \left\{ 1 + O\( \frac{1}{\log t^{-1}} \) \right\} \qquad (0 < t \le 1/2).
\end{equation}

The following result is due to de Bruijn \cite{deB}.
\begin{lemma} \label{debruijn}
Uniformly for $x \ge y \ge 2$ we have
\[
\log \Psi(x,y) = \frac{\log x}{\log y}\,G\Big(\frac{y}{\log x} \Big) \left\{ 1 + O \( \frac{1}{\log y} + \frac{1}{\log \log 2x} \) \right\}.
\]
\end{lemma}
For smaller values of $y$, we need the following result of Ennola \cite{Enn}.
\begin{lemma} \label{ennola}
Uniformly for $2 \le y \le \sqrt{\log x \log \log x}$ we have
\[
\Psi(x,y) = \frac{1}{\pi(y)!} \prod_{p \le y} \frac{\log x}{\log p}
\left\{ 1 + O \( \frac{y^2}{\log x \log y} \) \right\},
\]
where $\pi(y) = |\{ p \le y \}|$.
\end{lemma}

For any finite set of primes $S$, let $\cO_S^*$ denote the group of $S$-units in $\mathbb{Q}^*$; that is,
\[
\cO_S^* = \{a/b \in \mathbb{Q}^* : p\mid ab \Rightarrow p \in S\}.
\]
The next statement is a special case of a more general result of Evertse on
solutions to $S$-unit equations (see \cite[Theorem 3]{Ev}).
\begin{lemma} \label{evertse}
Given $a_1 \dots a_n \in \mathbb{Q}^*$ and a finite set of primes $S$
of cardinality $|S| = s$, the $S$-unit equation
\[
a_1 u_1 + \dots + a_n u_n = 1 \qquad (u_1, \dots , u_n \in \cO_S^*)
\]
has at most $(2^{35}n^2)^{n^3 s}$ solutions $(u_1, \dots , u_n)$
with $\sum_{j \in \cJ} a_ju_j \neq 0$ for every nonempty subset $\cJ \subseteq \{1 , \dots , n\}$.
\end{lemma}

To get a better handle on productsets of smooth numbers, we shall apply
the following technical lemma.

\begin{lemma} \label{tech}
We have
\[
\Psi(x^2/y , y) \le |\cS(x,y) \cdot \cS(x,y) | \le \Psi(x^2, y) \qquad (x \ge y \ge 2).
\]
\end{lemma}
\begin{proof}
It is easy to see that $\cS(x,y)\cdot\cS(x,y)\subseteq\cS(x^2,y)$, which yields
the second inequality. For the first inequality, it suffices to show that $\cS(x^2/y,y)$
is contained in the productset $\cS(x,y)\cdot\cS(x,y)$.
To this end, let $n\in\cS(x^2/y,y)$, and let $d$ be the largest divisor of $n$
that does not exceed~$x$. Note that $\max\{P^+(d),P^+(n/d)\}\le y$.
There are three possibilities for the number $d$:

\begin{itemize}
\item[$(i)$] $d>x/y$;

\item[$(ii)$] $d=n\le x/y$;

\item[$(iii)$] $d\le x/y$ and $d<n$.
\end{itemize}
In case $(i)$ we have $n/d\le x$, hence we can write $n=d\cdot(n/d)$ where $d$ and $n/d$ both lie in $\cS(x,y)$; this shows that $n\in\cS(x,y)\cdot\cS(x,y)$ as required.  In case $(ii)$ the number $n$ lies in the set $\cS(x/y,y)$, which is a subset of $\cS(x,y)\cdot\cS(x,y)$.  To finish the proof, we need only show that the case $(iii)$ is not possible. Indeed, suppose $d\le x/y$ and $d<n$, and let $p$ be any prime factor of $n/d$; then $p\le P^+(n/d)\le y$, $dp\mid n$, and $dp\le x$, which contradicts the maximal property of $d$.
\end{proof}

\section{Small values of $y$}

\begin{theorem} \label{smallsum}
There is an absolute constant $c>0$ such that the estimate
\[
|\cS(x,y) + \cS(x,y)| \sim \frac{1}{2} \Psi(x,y)^2 \qquad (x \to \infty)
\]
holds uniformly for $2 \le y \le c \log x$.
\end{theorem}

\begin{proof}
We have
\[
\Psi(x,y)^2=|\cS(x,y)|^2=\sum_{n\in
\cS(x,y)+\cS(x,y)}~\sum_{\substack{m_1,m_2\in\cS(x,y)\\m_1+m_2=n}}1.
\]
Using the Cauchy inequality it follows that
\[
\Psi(x,y)^4\le|\cS(x,y)+\cS(x,y)|\cdot|\cT|,
\]
where $\cT$ is the set of quadruples $(m_1,m_2,m_3,m_4)$ with
entries in $\cS(x,y)$ such that $m_1+m_2=m_3+m_4$.  It is easy to
see that there are precisely $2\,\Psi(x,y)^2-\Psi(x,y)$ quadruples
in $\cT$ for which $m_1=m_3$ or $m_1=m_4$.  Let $\cT^*$ be the set
of quadruples in $\cT$ with $m_1\ne m_3$ and $m_1\ne m_4$ (thus,
$m_2\ne m_3$ and $m_2\ne m_4$ as well).  If we put $a_1=a_2=1$ and
$a_3=-1$, the equation $m_1+m_2=m_3+m_4$ becomes
\begin{equation}
\label{eq:Suniteqn} a_1u_1+a_2u_2+a_3u_3=1,
\end{equation}
where
\begin{equation}
\label{eq:m/m} u_1=\frac{m_1}{m_4}\,,\qquad
u_2=\frac{m_2}{m_4} \text{ and } u_3=\frac{m_3}{m_4}\,.
\end{equation}
Let $S$ be the set of primes $p\le y$, and let $\cO_S^*$ be the
group of $S$-units in $\mathbb{Q}^*$. According to
Lemma~\ref{evertse}, there are at most
$(2^{35}\,9)^{27\pi(y)}$ solutions to the \text{$S$-unit}
equation~\eqref{eq:Suniteqn} with $u_j\in\cO_S^*$, $j=1,2,3$, and
$\sum_{j\in\cJ}a_ju_j\ne 0$ for each nonempty subset
$\cJ\subseteq\{1,2,3\}$.  On the other hand, for every fixed
solution $(u_1,u_2,u_3)$ to~\eqref{eq:Suniteqn} there are at most
$\Psi(x,y)$ quadruples $(m_1,m_2,m_3,m_4)$ in $\cT^*$ for
which~\eqref{eq:m/m} holds (since each choice of $m_4\in\cS(x,y)$
determines $m_1,m_2,m_3$ uniquely).  Putting everything together,
it follows that the bound
\[
\Psi(x,y)^4\le|\cS(x,y)+\cS(x,y)|
\cdot\(2\,\Psi(x,y)^2-\Psi(x,y)+\exp(c_1y/\log y)\Psi(x,y)\)
\]
holds with some absolute constant $c_1>0$. Taking into account the
trivial upper bound
\[
|\cS(x,y)+\cS(x,y)|\le\frac12\(\Psi(x,y)^2+\Psi(x,y)\),
\]
it suffices to show that there is an absolute constant $c>0$ such
that for all sufficiently large $x$, we have
\begin{equation}
\label{eq:expcylogy-est} \exp(c_1y/\log
y)\le\Psi(x,y)^{1/2}\qquad(2\le y\le c\log x).
\end{equation}

For every sufficiently large integer $N$, Lemma~\ref{debruijn}
implies that:
\[
\log\Psi(x,y)\ge\frac{1}{2}\,\frac{\log x}{\log
y}\,G\Big(\frac{y}{\log x}\Big)\qquad(x\ge y>N)
\]
if $x$ is sufficiently large. Let $N\ge 2$ be fixed with this
property.  For every sufficiently small constant $c>0$ we also
have by~\eqref{tsmall}:
\[
G(t)\ge\frac12\,t\,\log t^{-1}\qquad(0<t\le c).
\]
Let $0<c\le e^{-8c_1}$ be fixed with this property.  Combining the
two bounds, we see that
\[
\log\Psi(x,y)\ge\frac{\log(1/c)}{4}\,\frac{y}{\log y}\ge
2c_1\,\frac{y}{\log y}\qquad(N<y\le c\log x)
\]
if $x$ is large enough; this implies~\eqref{eq:expcylogy-est} in
the range $N<y\le c\log x$.  For the smaller values of $y$ in the
range $2\le y\le N$, we simply observe that $\exp(c_1y/\log
y)=O(1)$, whereas
\[
\Psi(x,y)\ge \Psi(x,2)=1+\left\lfloor {\frac{\log x}{\log
2}}\right\rfloor \to\infty \qquad\text{as~}x\to\infty.
\]
Hence, \eqref{eq:expcylogy-est} also holds for these values of $y$
if $x$ is sufficiently large.  This completes the proof.
\end{proof}


\begin{theorem} \label{prodsmall}
Suppose that $y \ge 2$ and $y = o(\log x)$.  Then
\[
|\cS(x,y) \cdot \cS(x,y)| = \Psi(x,y)^{1 + o(1)}.
\]
\end{theorem}

\begin{proof}
By Lemma~\ref{tech} we have
\[
\Psi(x,y)\le\Psi(x^2/y,y)\le
|\cS(x,y)\cdot\cS(x,y)|\le\Psi(x^2,y),
\]
hence it suffices to show that $\Psi(x^2,y)=\Psi(x,y)^{1+o(1)}$ as
$x\to\infty$.

First, suppose that $2\le y\le\sqrt{\log x}$. By
Lemma~\ref{ennola} we have
\[
\Psi(x,y)\sim\frac{1}{\pi(y)!}\prod_{p\le y}\frac{\log x}{\log p}
\qquad(x\to\infty)
\]
and
\[
\Psi(x^2,y)\sim\frac{1}{\pi(y)!}\prod_{p\le y}\frac{\log x^2}{\log
p}\sim 2^{\pi(y)}\Psi(x,y)\qquad(x\to\infty).
\]
Since the inequality $\pi(y)!\le y^{\pi(y)}$ implies
\[
\Psi(x,y)\ge(1+o(1))\(\frac{\log x}{y\log y}\)^{\pi(y)}
\ge(1+o(1))\(\frac{2\sqrt{\log x}}{\log\log x}\,\)^{\pi(y)},
\]
it follows that $2^{\pi(y)}=\Psi(x,y)^{o(1)}$; thus,
$\Psi(x^2,y)=\Psi(x,y)^{1+o(1)}$ as required.

Next, suppose that $y>\sqrt{\log x}$ and $y=o(\log x)$ as
$x\to\infty$. Using Lemma~\ref{debruijn} together
with~\eqref{tsmall} we see that the estimate
\[
\log\Psi(z,y)=\frac{y}{\log y}\,\log\Big(\frac{\log
z}{y}\Big)\left\{1+O\(\frac{1}{\log((\log x)/y)}\)\right\}
\]
holds uniformly for all $z$ in the range $x\le z\le x^2$. Applying
this estimate with $z=x$ and with $z=x^2$, we derive that
$\Psi(x^2,y)=\Psi(x,y)^{1+o(1)}$ in this case as well.
\end{proof}

\section{Large values of $y$}

For values of $y$ exceeding any fixed power of $\log x$, we have:

\begin{theorem}  \label{largesum}
Suppose that $(\log y)/ \log \log x \to \infty$.  Then,
\[
|\cS(x,y) + \cS(x,y)| = \Psi(x,y)^{1 + o(1)} \qquad (x \to \infty).
\]
\end{theorem}

\begin{proof}
Using Lemma~\ref{debruijn}
and~\eqref{tlarge} we see that
\[
\log\Psi(x,y)\sim\frac{\log x}{\log y}\,G\Big(\frac{y}{\log
x}\Big)\sim\frac{\log x}{\log y}\,(\log y-\log\log x)\sim\log
x\qquad(x\to\infty),
\]
since $(\log \log x)/\log y \to 0$; that is,
\[
\Psi(x,y)=x^{1+o(1)}\qquad(x\to\infty).
\]
Using the trivial bounds
\[
\Psi(x,y)\le|\cS(x,y)+\cS(x,y)|\le 2x
\]
together with the previous estimate, we obtain the desired result.

\end{proof}

\begin{theorem} \label{largeprod}
Let $y / \log x \to \infty$.  Then,
\begin{equation}
\label{eq:stapler}
|\cS(x,y) \cdot \cS(x,y)| = \Psi(x,y)^{2 + o(1)} \qquad (x \to \infty).
\end{equation}
\end{theorem}

\begin{proof}
In the case that
$(\log y)/\log \log x \to \infty$, we can apply Theorem~\ref{largesum} 
together with~\eqref{energy} to obtain \eqref{eq:stapler} immediately.
Thus, we can assume that $\log y\asymp\log\log x$.  Since $y/\log x\to\infty$,
we derive from Lemma~\ref{debruijn} and \eqref{tlarge} the estimate
\begin{equation}
\label{eq:can_of_pens}
\log\Psi(x,y)=\frac{\log x}{\log y}\,\log\Big(\frac{y}{\log x}\Big)
\left\{1+o(1)\right\},
\end{equation}
whereas both $\log\Psi(x^2/y,y)$ and $\log\Psi(x^2,y)$ are of the size
$$
\frac{\log x}{\log y}\,\log\Big(\frac{y}{\log x}\Big)\{2+o(1)\}.
$$
Therefore,
$$
\Psi(x^2/y,y)=\Psi(x,y)^{2+o(1)}\qquad\text{and}\qquad
\Psi(x^2,y)=\Psi(x,y)^{2+o(1)},
$$
and the result follows from Lemma~\ref{tech}.
\end{proof}

\section{Intermediate values of $y$}

\begin{theorem} \label{k}
Suppose that $y = \kappa \log x$, where $\kappa > 0$ is fixed.  Then,
\begin{equation} \label{kprod}
 |\cS(x,y) \cdot \cS(x,y)| = \Psi(x,y)^{\alpha_k + o(1)}
\end{equation}
and
\begin{equation} \label{ksum}
|\cS(x,y) + \cS(x,y)| \ge \Psi(x,y)^{(4- \alpha_{\kappa})/2 + o(1)},
\end{equation}
where
\[
\alpha_{\kappa} = \frac{2 \log (1 + \kappa /2) + \kappa \log(1 + 2 / \kappa)}{\log(1 + \kappa) + \kappa \log(1 + 1 / k)}.
\]
\end{theorem}

\noindent{\sc Remark.} For every positive real number $\kappa$ we have $1 < \alpha_{\kappa} < 2$.  Also, $\alpha_{\kappa} \to 1$ as $\kappa \to 0^+$ and $\alpha_{\kappa} \to 2$ as $\kappa \to \infty$.

\begin{proof}
First note that \eqref{ksum} follows from combining \eqref{kprod} and \eqref{energy}.  It remains to prove \eqref{kprod}.  By Lemma~\ref{debruijn} we have
\[
\log\Psi(x,y)=\big(G(\kappa)+o(1)\big)\,\frac{\log x}{\log\log
x}\qquad(x\to\infty)
\]
and
\[
\log\Psi(x^2,y)=\big(2\,G(\kappa/2)+o(1)\big)\,\frac{\log
x}{\log\log x}\qquad(x\to\infty),
\]
where the functions implied by $o(1)$ depend only on $\kappa$.
Since $G$ is continuous it is also easy to see that
\[
\log\Psi(x^2/y,y)=\big(2\,G(\kappa/2)+o(1)\big)\,\frac{\log
x}{\log\log x}\qquad(x\to\infty).
\]
Using Lemma~\ref{tech}, the above estimates, and the fact
that $\alpha_\kappa=2\,G(\kappa/2)/G(\kappa)$, the result follows.
\end{proof}

\begin{theorem} \label{A}
Suppose that $y \asymp (\log x)^A$, where $A>2$ is fixed.  Then,
\[
|\cS(x,y) + \cS(x,y)| \le \Psi(x,y)^{\frac{A}{A-1} + o(1)} \qquad (x \to \infty).
\]
\end{theorem}
\begin{proof}
If $y\asymp(\log x)^A$ for some $A>1$,
then the estimate $\Psi(x,y)=x^{\frac{A-1}{A}+o(1)}$ follows immediately
from~\eqref{eq:can_of_pens}. Taking into account the trivial bound $|\cS(x,y)+\cS(x,y)|\le 2x$,
we obtain the stated result (which is nontrivial in the range $A>2$). 
\end{proof}

\end{document}